\documentclass[12pt,a4paper,leqno]{amsart}

\usepackage{geometry,amssymb,amsmath,latexsym,amsthm,enumerate,graphicx}

\usepackage{colortbl}
\usepackage{ytableau}
\usepackage{tikz}
\usepackage[all]{xy}
\usepackage{comment}
\usepackage{bbm}

\theoremstyle{thmit} 
\newtheorem{thm}{Theorem}[section]
\newtheorem{lem}[thm]{Lemma}

\usepackage{mathrsfs}

\theoremstyle{definition}
\newtheorem{remark}[thm]{Remark}

\newtheorem{example}[thm]{Example}

\makeatletter
\@namedef{subjclassname@2020}{%
  \textup{2020} Mathematics Subject Classification}
\makeatother

\numberwithin{equation}{section}

\allowdisplaybreaks

\begin{document} 

\title[Determinant formulas of Giambelli-type]{
New determinant formulas of Giambelli-type for Schur multiple zeta-functions
and their applications
}
\author {Kohji Matsumoto}
\author{Maki Nakasuji}


\thanks{This work was supported by Japan Society for the Promotion of Science, Grant-in-Aid for Scientific Research No. 22K03267 (K. Matsumoto)
and No. 22K03274 (M. Nakasuji).}

\maketitle

\begin{abstract}
In this article, we will prove the Giambelli formula for Schur multiple zeta-functions of extended shape which we call laced type,
using the combinatorial method of proving the Giambelli formula for Schur function by  E\v{g}ecio{\v g}lu and Remmel. 
Further we will obtain the Giambelli formula for Schur multiple zeta-functions of a certain skew type via the antipode on the set of quasi-symmetric functions.
Combining these two Giambelli-type formulas, we will have new identities among Schur multiple zeta-functions.
\end{abstract}

\bigskip
\section{Introduction}\label{sec-1}

Denote by $\mathbb{N}$, $\mathbb{C}$ be the set of all positive integers, and of all
complex numbers, respectively.

Let $\lambda=(\lambda_1,\ldots,\lambda_m)$ be a partition of $n\in\mathbb{N}$, 
and let $T(\lambda,X)$ the set of all Young tableaux of shape
$\lambda$ over a set $X$.    Let ${\rm SSYT}(\lambda)\subset T(\lambda,\mathbb{N})$ be 
the set of all
semi-standard Young tableaux of shape $\lambda$.
We write $M=(m_{ij})\in {\rm SSYT}(\lambda)$, where $m_{ij}$ denotes the number in the
$(i,j)$-component of the tableaux.

The Schur multiple zeta-function associated with $\lambda$, introduced by
Nakasuji, Phuksuwan and Yamasaki \cite{NPY18}, is defined by
\begin{align}\label{1-1}
\zeta_{\lambda}({\pmb s})=\sum_{M\in {\rm SSYT}(\lambda)}\frac{1}{M^{\pmb s}}=\sum_{M\in {\rm SSYT}(\lambda)}
w(s, M),
\end{align}
where ${\pmb s}=(s_{ij})\in T(\lambda,\mathbb{C})$ and $w(s, M)=\dfrac{1}{M^{\pmb s}}$ with
$$
M^{\pmb s}=\prod_{(i,j)\in\lambda}m_{ij}^{s_{ij}}\qquad
(M=(m_{ij})\in {\rm SSYT}(\lambda)).
$$

The series \eqref{1-1} converges absolutely if ${\pmb s}\in W_{\lambda}$ where 
\[
  W_\lambda =
\left\{{\pmb s}=(s_{ij})\in T(\lambda,\mathbb{C})\,\left|\,
\begin{array}{l}
 \text{$\Re(s_{ij})\ge 1$ for all $(i,j)\in \lambda \setminus C(\lambda)$ } \\[3pt]
 \text{$\Re(s_{ij})>1$ for all $(i,j)\in C(\lambda)$}
\end{array}
\right.
\right\}
\]
 with $C(\lambda)$ being the set of all corners of $\lambda$.

In \cite{NPY18}, several determinant expressions of $\zeta_{\lambda}({\pmb s})$ were
established.    In particular, here we quote two important determinant formulas for
$\zeta_{\lambda}({\pmb s})$.

For a partition $\lambda$, let $\lambda'=(\lambda_1',\ldots,\lambda'_{m'})$ be its
conjugate, and
we define two sequences of indices
$p_1, \ldots, p_N$ and $q_1, \ldots, q_N$ by $p_i=\lambda_i-i$ and $q_i=\lambda'_i-i$ for $1\leq i \leq N$
where $N$ is the number of the main diagonal entries of the Young diagram of $\lambda$.
We sometimes write $\lambda=(p_1, \ldots, p_N | q_1, \ldots, q_N)$, 
the Frobenius notation of $\lambda$ (see \cite[Section 1.1]{Mac95}).

Let 
$$T^{\rm{diag}}(\lambda,\mathbb{N})=\{T=(t_{ij})\in T(\lambda,\mathbb{N})\,|\,\text{$t_{ij}=t_{lm}$ if $j-i=m-l$}\}$$
and
$$W_{\lambda}^{\rm{diag}}=\{{\pmb s}=(s_{ij})\in W_{\lambda}\,|\,\text{$s_{ij}=s_{lm}$ if $j-i=m-l$}\},$$
that is, all the variables on the same diagonal lines (from the upper left to the
lower right) are the same.  
When ${\pmb s}\in W_{\lambda}^{\rm{diag}}$, we can introduce new variables 
$\pmb{z}=\{z_k\}_{k\in\mathbb{Z}}$ by the condition $s_{ij}=z_{j-i}$ (for all $i,j$),
and we may regard $\zeta_{\lambda}({\pmb s})$ as a function in variables 
$\{z_k\}_{k\in\mathbb{Z}}$.
We call the Schur multiple zeta-function associated with $\{z_k\}$ {\it content-parametrized Schur multiple zeta-function}, since $j-i$ is named ``content''.

Next, define the $r$-fold multiple zeta-function of Euler-Zagier type as
$$
\zeta_{EZ,r}(s_1,\ldots,s_r)=\sum_{0<m_1<\cdots<m_r}\frac{1}{m_1^{s_1}\cdots m_r^{s_r}}.
$$
This is nothing but the Schur multiple zeta-function attached to
$\lambda=(\underbrace{1, 1, \ldots, 1}_r)=(1^r)$.    The Jacobi-Trudi formula for $\zeta_{\lambda}(\pmb{s})$ is as
follows.

\begin{thm}\label{JT}
{\rm (\cite[Theorem 1.1 (2)]{NPY18})}
Let ${\pmb s}\in W_{\lambda}^{\rm{diag}}$. 
Then
\begin{align}\label{JT-formula}
\zeta_{\lambda}(\pmb{s})=\det(\zeta_{EZ,\lambda'_i-i+j}(z_{-j+1},z_{-j+2},\ldots,
z_{-j+(\lambda'_i-i+j)})_{1\leq i,j\leq m'},
\end{align}
here we understand that $\zeta_{EZ,\lambda'_i-i+j}(\cdots)=1\; ({\rm resp.} =0)$ if 
$\lambda'_i-i+j=0 \;({\rm resp.} <0)$.
\end{thm}

On the right-hand side of \eqref{JT-formula}, in the definition of the determinant, 
the product of Euler-Zagier multiple zeta-functions is to be given by the harmonic product.

\begin{remark}\label{meroconti}
In the original statement of the above theorem in \cite{NPY18}, the condition
$\Re(s_{\lambda'_i,i})=\Re (z_{-q_i})>1$ (for $1\leq i\leq m'$) is required.    However, 
as commented in \cite{NPY18}, \eqref{JT-formula} gives the meromorphic continuation of
$\zeta_{\lambda}(\pmb{s})$, as a function in $\pmb{z}$, to the whole $\pmb{z}$-space.
In this sense, we can state Theorem \ref{JT} without the condition 
$\Re(s_{\lambda'_i,i})=\Re (z_{-q_i})>1$.
\end{remark}

Next, the Giambelli formula for $\zeta_{\lambda}({\pmb s})$ proved in \cite{NPY18}
is as follows.

\begin{thm}\label{Giambelli0} 
{\rm (\cite[Theorem 4.5]{NPY18})}
Let  $\lambda$ be a partition whose
Frobenius notation is $\lambda=(p_1, \cdots , p_N | q_1, \cdots, q_N)$.
Assume ${\pmb s}\in W_{\lambda}^{\rm{diag}}$. 
Then we have
\begin{equation}\label{expij}
\zeta_{\lambda}({\pmb s}) = \det(\zeta_{i,j})_{1 \leq i,j \leq N},
\end{equation}
where $\zeta_{i,j}=\zeta_{(p_i+1, 1^{q_j})} ({\bf s}_{i,j}^F)$ with 
${\bf s}_{i,j}^F=
\ytableausetup{boxsize=normal}
  \begin{ytableau}
   z_0 & z_1 & z_2 &\cdots & z_{p_i}\\
   z_{-1}\\
   \vdots \\
    z_{-q_j}
  \end{ytableau}\in W_{(p_i+1, 1^{q_j})}.
$
 \end{thm}

\begin{remark}
On the right-hand side of \eqref{expij}, to calculate the determinant, it is necessary
to define the products among $\zeta_{i,j}$.    This products are given by the harmonic
products.    In fact, more generally, if $\pmb{s}\in W_{\lambda}^{\rm diag}$ and
$\pmb{s}'\in W_{\lambda'}^{\rm diag}$, then we can define the harmonic product 
$\zeta_{\lambda}(\pmb{s})\ast \zeta_{\lambda'}(\pmb{s}')$ by applying Theorem
\ref{JT} and reducing to the harmonic product among multiple zeta-functions of
Euler-Zagier type.
\end{remark}

It is the aim of the present paper to show two new variants of Theorem \ref{Giambelli0},
and report their applications.

First, in Section \ref{sec-2}, we will state a generalization of Theorem \ref{Giambelli0}
to Schur multiple zeta-functions attached to extended tableaux, named ``laced tableaux'' 
(Theorem \ref{Giambelli1}).
The proof of Theorem \ref{Giambelli1}
will be described in Section \ref{sec-5}.

In \cite{NPY18}, Theorem \ref{Giambelli0} is just reduced to the Giambelli-type formula
given in Nakagawa et al. \cite{NNSY01}.    Our proof of Theorem \ref{Giambelli1} is different;
it is based on a paper of E{\v g}ecio{\v g}lu and Remmel \cite{EgeRem88}.    
Therefore especially
our argument includes an alternative proof of Theorem \ref{Giambelli0}.

The second variant of the Giambelli formula in the present paper
will be shown in Section \ref{sec-3}, 
which is for Schur multiple zeta-functions attached to tableaux of skew type (Theorem 
\ref{Giambelli2}).

In our previous article \cite{MatNak21}, we proved that Schur multiple zeta-functions
attached to tableaux of anti-hook type can be expressed in terms of multiple zeta-functions of
Euler-Zagier type and their star variant (\cite[Theorem 3.2]{MatNak21}), or in terms
of zeta-functions of root systems (\cite[Theorem 4.1]{MatNak21}).
Similar results were also shown in the case of hook type (\cite[Section 3]{MatNakPrep}).
In the latter paper, using the Giambelli formula (Theorem \ref{Giambelli0}), 
the results were further extended to more general content-parametrized Schur multiple
zeta-functions (\cite[Section 4]{MatNakPrep}).
Now, Theorem \ref{Giambelli2} in the present paper gives a Giambelli-type formula for
tableaux of skew type.
Therefore, we can extend the results in \cite{MatNak21}
to the case of content-parametrized Schur multiple zeta-functions attached to tableaux of
more general skew type (see Remark \ref{R3-6}).

In Section \ref{sec-4}, as another application of our theorems,
we will consider the tableaux of the following form:
$$
\ytableausetup{boxsize=normal}
  \begin{ytableau}
   \none & \none & \none & \none & \none &  \\
   \none & \none & \none & \none & \none & \vdots \\
   \none & \none & \none & \none & \none &  \\
   \none & \none & \none &  & \cdots &  \\
   \none & \none & \none & \vdots &  & \vdots \\
     & \cdots &  &  & \cdots &
  \end{ytableau}.
$$
We will apply both Theorem \ref{Giambelli1} and Theorem \ref{Giambelli2} to 
this type of tableaux, and will 
deduce new identities among Schur multiple zeta-functions.

\begin{example}\label{example-1}
For 
$$
{\pmb s}=
\ytableausetup{boxsize=normal}
  \begin{ytableau}
 \none & \none & \none & \alpha_2 \\
  \none & \none & \none &  \alpha_1\\
 \none & \none &  c_0 & c_1  \\
   \beta_2  & \beta_1 & c_{-1} & c_0
  \end{ytableau},
$$
where we call this tableau shape $\theta$. Then we have the following identity:
$$
\zeta_{\theta}({\pmb s})=
\left|
\begin{matrix}
\;
\ytableausetup{boxsize=normal}
  \begin{ytableau}
 \none & \none & \none & \alpha_2 \\
  \none & \none & \none &  \alpha_1\\
 \none & \none &  c_0 & c_1  \\
   \beta_2  & \beta_1 & c_{-1} & \none
  \end{ytableau}
  &
  \ytableausetup{boxsize=normal}
  \begin{ytableau}
  \none & \alpha_2 \\
  \none &  \alpha_1\\
   c_0 & c_1  
   \end{ytableau}\;
  \vspace{2mm}\\
  \ytableausetup{boxsize=normal}
  \begin{ytableau}
  \none & \none &  c_0  \\
   \beta_2  & \beta_1 & c_{-1}
  \end{ytableau}
  &
  \ytableausetup{boxsize=normal}
  \begin{ytableau}
 c_0
  \end{ytableau}  
\end{matrix}
\right|
=
\left|
\begin{matrix}\;
\ytableausetup{boxsize=normal}
  \begin{ytableau}
 \none & \none & \none & \alpha_2 \\
  \none & \none & \none &  \alpha_1\\
 \none & \none &  \none & c_1  \\
   \beta_2  & \beta_1 & c_{-1} & c_0
  \end{ytableau}
  &
  \ytableausetup{boxsize=normal}
  \begin{ytableau}
 \alpha_2 \\
 \alpha_1\\
 c_1  \\
 c_0
   \end{ytableau}
   \; 
  \vspace{2mm}\\
  %
  \ytableausetup{boxsize=normal}
  \begin{ytableau}
   \beta_2  & \beta_1 & c_{-1} & c_0
  \end{ytableau}
    &
  \ytableausetup{boxsize=normal}
  \begin{ytableau}
 c_0
  \end{ytableau}  
\end{matrix}
\right|,
$$
where in the middle- and right-hand sides, each tableau $\widetilde{\pmb s}$ in the determinant means $\zeta_{\widetilde{\lambda}}(\widetilde{\pmb s})$, for short,
with $\widetilde{\lambda}$ being the shape of $\widetilde{\pmb s}$.
Throughout this article, we use this convention if there is no confusion.
\end{example}


\begin{remark}
The arguments and results in the present article can be directly extended to the more
general framework of quasi-symmetric functions. 
\end{remark}

\section{A Giambelli-type formula for Schur multiple zeta-functions
of laced type}\label{sec-2}

Let $\lambda_*$, $\lambda^*$ be two tableaux. 
Let $\lambda$ be another tableau, and now construct a new tableau $\widetilde{\lambda}$ 
by pasting these three tableaux in the following way:

(i) The right-edge of the most north-east box of $\lambda_*$ is pasted to the 
left-edge of the most south-west box of $\lambda$,

(ii) The bottom-edge of the most south-west box of $\lambda^*$ is pasted to the
top-edge of the most north-east box of $\lambda$.

We  denote this $\tilde{\lambda}$ by
$\widetilde{\lambda}=[\lambda_*\mid\lambda\mid\lambda^*]$, and
call it  {\it laced tableau}.
For example, 
the following figure (which indeed looks laced) is the one of laced tableaux, where the  colored part is $\lambda$  in $\widetilde{\lambda}$:
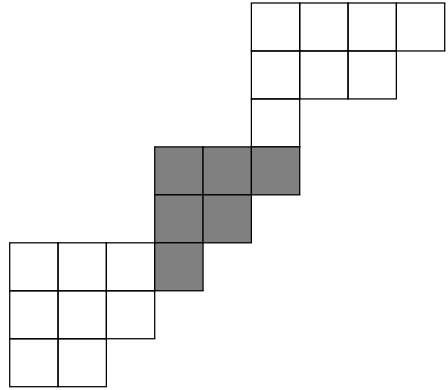
$$
\begin{ytableau}
    \none & \none & \none & \none & \none & & & & \\
    \none & \none & \none & \none & \none & & & \\
    \none & \none & \none & \none & \none & \\
    \none & \none & \none & *(gray)  & *(gray) & *(gray)\\
    \none & \none & \none & *(gray) & *(gray) \\
    &  & & *(gray) \\
    &  &  \\
        &   
  \end{ytableau}
  $$

In particular, if
\begin{align*}
\ytableausetup{boxsize=normal}
\lambda_*=
  \begin{ytableau}
    \none &  & \cdots &  
  \end{ytableau}\;,
\qquad \lambda^*=
  \begin{ytableau}
   \none &     \\
   \none & \vdots  \\  
   \none &
  \end{ytableau}\;,
\end{align*}
then the figure of $\widetilde{\lambda}$ looks like a bird whose body $\lambda$ has two wings
$\lambda_*$ and $\lambda^*$.   This $\widetilde{\lambda}$ is called a tableau of
{\it winged type}, which was first named by E{\v g}ecio{\v g}lu and Remmel (\cite{EgeRem88}).

Let $\pmb{s}\in W_{\lambda}, \pmb{s}_*\in W_{\lambda_*}, \pmb{s}^*\in W_{\lambda^*}$,
and define $\widetilde{\pmb{s}}\in W_{\widetilde{\lambda}}$,
consisting of $\pmb{s}, \pmb{s_*}$ and $\pmb{s^*}$, which we denote
$\widetilde{\pmb{s}}=[\pmb{s}_*\mid\pmb{s}\mid\pmb{s}^*]$.
Our main target in the present paper is the zeta-function
$\zeta_{\widetilde{\lambda}}(\widetilde{\pmb{s}})$.

On the middle part $\lambda$, we use the same notation, and the same
assumption $\pmb{s}\in W_{\lambda}^{\rm diag}$ as in the statement of
Theorem \ref{Giambelli0}.    

Recall $\zeta_{ij}$ and $\pmb{s}_{ij}^F$ in the statement of Theorem \ref{Giambelli0}
(for $1\leq i,j\leq N$).   Let
$\lambda_{ij}=(p_i+1,1^{q_j})$, 
\begin{align*}
\widetilde{\lambda}_{ij}=
\begin{cases}
[\lambda_*\mid \lambda_{11} \mid \lambda^*] & {\rm for}\; i=j=1,\\
[\emptyset\mid \lambda_{1j} \mid \lambda^*] & {\rm for}\; i=1, 1<j\leq N\\
[\lambda_*\mid \lambda_{i1} \mid \emptyset] & {\rm for}\; 1<i\leq N, j=1\\
\lambda_{ij} & {\rm for}\; 1<i\leq N, 1<j\leq N
\end{cases}
\end{align*}
\begin{align*}
\widetilde{\pmb{s}}_{ij}^F=
\begin{cases}
[\pmb{s}_*\mid \pmb{s}_{11}^F \mid \pmb{s}^*] & {\rm for}\; i=j=1,\\
[\emptyset\mid \pmb{s}_{1j}^F \mid \pmb{s}^*] & {\rm for}\; i=1, 1<j\leq N\\
[\pmb{s}_*\mid \pmb{s}_{i1}^F \mid \emptyset] & {\rm for}\; 1<i\leq N, j=1\\
\pmb{s}_{ij}^F & {\rm for}\; 1<i\leq N, 1<j\leq N
\end{cases}
\end{align*}
and define the zeta-function associated with $\widetilde{\lambda}_{ij}$ by
$$
\widetilde{\zeta}_{i,j}=\zeta_{\widetilde{\lambda}_{ij}}
(\widetilde{\pmb{s}}_{ij}^F).
$$

\begin{example}For $\lambda=(4, 3, 2, 1)$, $\lambda_*=(2)$, $\lambda^*=(1,1)$,
$$
{\pmb s}=
\ytableausetup{boxsize=normal}
  \begin{ytableau}
c_0 & c_1 &c_2 & c_3 \\
c_{-1} & c_0 &c_1\\
c_{-2} & c_{-1}\\
c_{-3}
\end{ytableau}\in W_{\lambda}, 
\quad
{\pmb s_*}=
\ytableausetup{boxsize=normal}
  \begin{ytableau}
\beta_2 & \beta_1
\end{ytableau}\in W_{\lambda_*}, 
\quad
{\pmb s^*}=
\ytableausetup{boxsize=normal}
  \begin{ytableau}
\alpha_2\\
\alpha_1
\end{ytableau}\in W_{\lambda^*}, 
$$
where we say $s_{ij}=c_{j-i}$, and
$$
\widetilde{\pmb{s}}=[\pmb{s}_*\mid\pmb{s}\mid\pmb{s}^*]=
\ytableausetup{boxsize=normal}
  \begin{ytableau}
\none & \none & \none & \none & \none & \alpha_2  \\
\none & \none  & \none & \none & \none & \alpha_1  \\
\none & \none & c_0 & c_1 &c_2 & c_3 \\
\none & \none  & c_{-1} & c_0 &c_1\\
\none & \none & c_{-2} & c_{-1}\\
\beta_2 & \beta_1 & c_{-3} 
\end{ytableau}.
$$
This tableau is ``winged type". In Frobenius notation, $\lambda=(3,1|3,1)$ with $N=2$.
\begin{align*}
\widetilde{{\pmb s}_{11}}^F & =[{\pmb s}_{*} | {\pmb s}_{11}^F | {\pmb s}^{*}]=
\ytableausetup{boxsize=normal}
  \begin{ytableau}
\none & \none & \none & \none & \none & \alpha_2  \\
\none & \none  & \none & \none & \none & \alpha_1  \\
\none & \none & c_0 & c_1 &c_2 & c_3 \\
\none & \none  & c_{-1}\\
\none & \none & c_{-2} \\
\beta_2 & \beta_1 & c_{-3} 
\end{ytableau},\\
\widetilde{{\pmb s}_{12}}^F & =[\emptyset\mid \pmb{s}_{12}^F \mid \pmb{s}^*] =
\ytableausetup{boxsize=normal}
  \begin{ytableau}
 \none & \none & \none & \alpha_2  \\
 \none & \none & \none & \alpha_1  \\
 c_0 & c_1 &c_2 & c_3 \\
 c_{-1}\\
\end{ytableau},\\
\widetilde{{\pmb s}_{21}}^F & =[{\pmb s}_* \mid \pmb{s}_{21}^F \mid \emptyset] =
\ytableausetup{boxsize=normal}
  \begin{ytableau}
\none & \none  & c_0 & c_1\\
\none & \none  & c_{-1}\\
\none & \none & c_{-2} \\
\beta_2 & \beta_1 & c_{-3} 
\end{ytableau},\\
\widetilde{{\pmb s}_{22}}^F & =[\emptyset \mid \pmb{s}_{22}^F \mid \emptyset] =
\ytableausetup{boxsize=normal}
  \begin{ytableau}
 c_0 & c_1\\
 c_{-1}\\
\end{ytableau}.\\
\end{align*}
\end{example}

The following is one of the main results in the present paper, which
is the Giambelli formula for $\zeta_{\widetilde{\lambda}}(\widetilde{\pmb s})$.

\begin{thm}\label{Giambelli1} 
Let  $\lambda$ be a partition, and write its  
Frobenius notation as $\lambda=(p_1, \cdots , p_N | q_1, \cdots, q_N)$.
Define $\widetilde{\lambda}$, $\widetilde{\pmb s}$,
$\widetilde{\lambda}_{ij}$, $\widetilde{\pmb s}_{ij}^F$
etc.~as above.    
Assume ${\pmb s}\in W_{\lambda}^{\rm{diag}}$. 
Then we have
\begin{equation}\label{expij-1}
\zeta_{\widetilde{\lambda}}(\widetilde{\pmb s}) = 
\det(\widetilde{\zeta}_{i,j})_{1 \leq i,j \leq N}.
\end{equation}
 \end{thm}
 
\begin{example}
For $\lambda=(2,1, 0 | 2, 1, 0)$ in Frobenius notation with $N=3$, any ${\pmb s}_*\in W_{\lambda_*}$ and ${\pmb s}^*\in W_{\lambda^*}$, we have
\begin{equation}\label{Ex210}
\ytableausetup{boxsize=normal}
  \begin{ytableau}
   \none & \none & \none & {\pmb s}^*  \\
   \none & c_{0} & c_{1} &  c_{2} \\
  \none  & c_{-1} & c_{0} & c_{1}  \\ 
 {\pmb s}_*    & c_{-2} & c_{-1} & c_{0} 
  \end{ytableau}
  =
  \left|
  \begin{matrix}
\ytableausetup{boxsize=normal}
  \begin{ytableau}
   \none & \none & \none & {\pmb s}^*  \\
   \none & c_{0}  &  c_{1} &  c_{2} \\
  \none  &  c_{-1}  \\ 
  {\pmb s}_*   & c_{-2}   
  \end{ytableau}

&

\ytableausetup{boxsize=normal}
  \begin{ytableau}
   \none & \none & \none & {\pmb s}^*  \\
   \none & c_{0}  &  c_{1} &  c_{2} \\
  \none  &  c_{-1}  
   \end{ytableau}

&

\ytableausetup{boxsize=normal}
  \begin{ytableau}
   \none & \none & \none & {\pmb s}^*  \\
   \none & c_{0}  &  c_{1} &  c_{2} 
   \end{ytableau}

\\
\\

\ytableausetup{boxsize=normal}
  \begin{ytableau}
   \none & c_{0}  &  c_{1} \\
  \none  &  c_{-1}  \\ 
  {\pmb s}_*   & c_{-2}   
  \end{ytableau}

&

\ytableausetup{boxsize=normal}
  \begin{ytableau}
   \none & c_{0}  &  c_{1} \\
  \none  &  c_{-1}  
   \end{ytableau}

&

\ytableausetup{boxsize=normal}
  \begin{ytableau}
   \none & c_{0}  &  c_{1} 
   \end{ytableau}

\\
\\

\ytableausetup{boxsize=normal}
  \begin{ytableau}
   \none & c_{0}  \\
  \none  &  c_{-1}  \\ 
  {\pmb s}_*   & c_{-2}   
  \end{ytableau}

&

\ytableausetup{boxsize=normal}
  \begin{ytableau}
   \none & c_{0}  \\
  \none  &  c_{-1}  
   \end{ytableau}

&

\ytableausetup{boxsize=normal}
  \begin{ytableau}
   \none & c_{0} 
   \end{ytableau}

\end{matrix}
\right|.
\end{equation}
\end{example}

When $\lambda_*=\lambda^*=\emptyset$, then $\widetilde{\lambda}=\lambda$, 
$\widetilde{\lambda}_{ij}=\lambda_{ij}$, 
$\widetilde{\pmb{s}}_{ij}^F=\pmb{s}_{ij}^F$, 
and $\widetilde{\zeta}_{i,j}=\zeta_{ij}$.
Therefore it is clear that Theorem \ref{Giambelli1} is a generalization of
Theorem \ref{Giambelli0}.

We will prove Theorem \ref{Giambelli1} in Section \ref{sec-5}.
As mentioned in Section \ref{sec-1}, our proof is based on the idea of
E{\v g}ecio{\v g}lu and Remmel \cite{EgeRem88}, different from the idea of the
proof of Theorem \ref{Giambelli0} given in \cite{NPY18}.
Especially
our argument includes an alternative proof of Theorem \ref{Giambelli0}.

\begin{remark}
We can extend Theorem \ref{Giambelli1} to that for Hurwitz type,
which is the case when the denominator of \eqref{1-1} includes shifting parameters.
Indeed, such a generalization to that of winged type has been obtained in \cite{MNHurwitz}.
\end{remark}

\section{A Giambelli-type formula for Schur multiple zeta-functions
of skew type}\label{sec-3}

In this section we give another main result of the present paper.

Recall the definition of Schur multiple zeta-functions of skew type.
A skew Young diagram $\theta$ is a diagram obtained as a set difference of two 
Young diagrams of partitions $\lambda$ and $\mu$, satisfying $\mu\subset\lambda$.
Then the Schur multiple zeta-function attached to $\theta$, which we denote by
$\zeta_{\theta}(\pmb{s})$, is defined similarly as \eqref{1-1}, just $\lambda$
replaced by $\theta$.
We also use the notation $T(\theta,X)$, $C(\theta)$ analogously to
$T(\lambda,X)$, $C(\lambda)$, respectively.

Let ${\it Qsym}$ be the set of all quasi-symmetric functions, which is an algebra
with the harmonic product $*$.    In fact, ${\it Qsym}$ has a Hopf algebra
structure with coproduct $\Delta$ (see Hoffman \cite{Hof15}).
The antipode $\mathscr{S}$ is an automorphism of ${\it Qsym}$ and satisfies 
$\mathscr{S}^2={\rm id}$.
As for the explicit form of $\mathscr{S}$, see \cite[Theorem 3.1]{Hof15}.

The quasi-symmetric function of Schur type attached to $\lambda$ is defined by
\begin{align}\label{3-1}
S_{\lambda}(\boldsymbol{\gamma})=\sum_{M=(m_{ij}) \in {\rm SSYT}(\lambda)}
\prod_{(i,j)\in\lambda}t_{m_{ij}}^{\gamma_{ij}},
\end{align}
where $t_{m_{ij}}$ are variables, and
$\boldsymbol{\gamma}=(\gamma_{ij})\in T(\theta,\mathbb{N})$.

Assume that $\lambda$ is a partition whose
Frobenius notation is $\lambda=(p_1, \cdots , p_N | q_1, \cdots, q_N)$.
If $\boldsymbol{\gamma}=(\gamma_{ij})\in T^{\rm{diag}}(\lambda,\mathbb{N})$,
then we have (see \cite[Theorem 5.4]{NPY18})
\begin{equation}\label{S-expij}
S_{\lambda}(\boldsymbol{\gamma}) = 
\det(S_{(p_i, 1^{q_j})} (\boldsymbol{\gamma}_{ij}^F))_{1 \leq i,j \leq N},
\end{equation}
with
$\boldsymbol{\gamma}_{ij}^F=
\ytableausetup{boxsize=normal}
  \begin{ytableau}
   c_0 & c_1 & c_2 &\cdots & c_{p_i}\\
   c_{-1}\\
   \vdots \\
    c_{-q_j}
  \end{ytableau}\in W_{(p_i, 1^{q_j})},
$
where $c_k$ are defined by $\gamma_{i,j}=c_{j-i}$ (for all $i,j$).
This \eqref{S-expij} is the Giambelli formula for $S_{\lambda}(\boldsymbol{\gamma})$.
Applying $\mathscr{S}$ to the both sides of \eqref{S-expij}, we obtain
\begin{align}\label{SS-expij}
\mathscr{S}(S_{\lambda}(\boldsymbol{\gamma})) = 
\mathscr{S}(\det(S_{(p_i, 1^{q_j})} (\boldsymbol{\gamma}_{ij}^F))_{1 \leq i,j \leq N}).
\end{align}

For any (skew) diagram $\theta$,
by $\theta^{\#}$ we denote the transpose of $\theta$ with respect to the anti-diagonal.
Similarly denote by $\boldsymbol{\gamma}^{\#}$ the anti-diagonal transpose of
$\boldsymbol{\gamma}$.    Then it is known that
\begin{align}\label{3-2}
\mathscr{S}(S_\theta (\boldsymbol{\gamma}))=(-1)^{|\theta|}
S_{\theta^{\#}}(\boldsymbol{\gamma}^{\#})
\end{align}
(see \cite[Theorem 5.9]{NPY18}).
Therefore the left-hand side of \eqref{SS-expij} is 
$(-1)^{|\lambda|}
S_{\lambda^{\#}}(\boldsymbol{\gamma}^{\#})$.
On the other hand, since $\mathscr{S}$ is an automorphism, the right-hand side of
\eqref{SS-expij} is 
\begin{align*}
&=\det(\mathscr{S}(S_{(p_i,1^{q_j})}(\gamma_{ij}^F))_{1\leq i,j\leq N}\\
&=\det((-1)^{p_i+q_j+1}S_{(p_i,1^{q_j})^{\#}}(\gamma_{ij}^{F\#}))_{1\leq i,j\leq N}\\
&=(-1)^{|\lambda|}\det(S_{(p_i,1^{q_j})^{\#}}(\gamma_{ij}^{F\#}))_{1\leq i,j\leq N}.
\end{align*}
Therefore, now, we can rewrite \eqref{SS-expij} as
\begin{align}\label{SSS-expij}
S_{\lambda^{\#}}(\boldsymbol{\gamma}^{\#})=
\det(S_{(p_i,1^{q_j})^{\#}}(\gamma_{ij}^{F\#}))_{1\leq i,j\leq N}.
\end{align}

Let
$$
I_{\theta}=\{\boldsymbol{\gamma}=(\gamma_{ij})\in T(\theta,\mathbb{N})\mid
\gamma_{ij}\geq 2\;{\rm for \;all}\; (i,j)\in C(\theta)\}.
$$
Let $\rho\phi^{-1}$ be the mapping defined in \cite{Hof15} 
(see also \cite[Section 5]{NPY18}), which sends
$S_{\theta}(\boldsymbol{\gamma})$ to $\zeta_{\theta}(\boldsymbol{\gamma})$
when $\boldsymbol{\gamma}\in I_{\theta}$ (see \cite[Lemma 5.1]{NPY18}).
\begin{remark}
We can say that the antipode $\mathscr{S}$ for multiple zeta functions can be defined by 
$\zeta_{\theta^{\#}}:=\rho \phi^{-1}(\mathscr{S}(S_{\theta}))$ as in 
the following commutative diagram:
\begin{equation}
\xymatrix{
S_{\theta}\ar[r]^-{\mathscr{S}}\ar[d]_-{\rho\phi^{-1}}\ar@{}[rd]|{\text{\rotatebox{180}{
$\circlearrowright$}}}&S_{\theta^{\#}}\ar[d]^-{\rho\phi^{-1}}\\
\zeta_{\theta}\ar[r]_-{\text{def}}&\zeta_{\theta^{\#}}
}
\end{equation}
\end{remark}

Applying $\rho\phi^{-1}$ to \eqref{SSS-expij}, we obtain
\begin{align}\label{zeta-ij}
\zeta_{\lambda^{\#}}(\boldsymbol{\gamma}^{\#})=
\rho\phi^{-1}(\det(S_{(p_i,1^{q_j})^{\#}}(\gamma_{ij}^{F\#}))_{1\leq i,j\leq N})
\end{align}
when $\boldsymbol{\gamma}^{\#}\in I_{\lambda^{\#}}$.
Since $\phi$ is an isomorphism (\cite[Theorem 2.2]{Hof15}) and $\rho$ is a 
homomorphism (\cite[p. 352]{Hof15}), we see that $\rho\phi^{-1}$ is also
a homomorphism.   Therefore the right-hand side of \eqref{zeta-ij} is
\begin{align*}
&=\det(\rho\phi^{-1}(S_{(p_i,1^{q_j})^{\#}}(\gamma_{ij}^{F\#})))_{1\leq i,j\leq N}\\
&=\det(\zeta_{(p_i,1^{q_j})^{\#}}(\gamma_{ij}^{F\#}))_{1\leq i,j\leq N}.
\end{align*}
Thus we now conclude the following result, which is the Giambelli-type formula
for Schur multiple zeta-functions of skew type.

\begin{thm}\label{Giambelli2}
Let  $\lambda$ be a partition whose
Frobenius notation is $\lambda=(p_1, \cdots , p_N | q_1, \cdots, q_N)$.
Let $\boldsymbol{\gamma} \in T(\lambda,\mathbb{N})$, and
assume $\boldsymbol{\gamma}^{\#}\in I_{\lambda^{\#}}$.
Then we have
\begin{align}\label{Giambelli2-formula}
\zeta_{\lambda^{\#}}(\boldsymbol{\gamma}^{\#})=
\det(\zeta_{(p_i,1^{q_j})^{\#}}(\gamma_{ij}^{F\#}))_{1\leq i,j\leq N}.
\end{align}
\end{thm}

\begin{example}
For $\lambda=(4, 3, 3, 2)$ and
$$
{\pmb \gamma}=
\ytableausetup{boxsize=normal}
  \begin{ytableau}
  c_0 & c_1 & c_2 & c_{3}  \\
  c_{-1} & c_{0} & c_{1} \\
 c_{-2}    & c_{-1} & c_{0}  \\ 
 c_{-3}    & c_{-2}  
  \end{ytableau},
$$ applying Theorem \ref{Giambelli2} gives
$$
\ytableausetup{boxsize=normal}
  \begin{ytableau}
   \none & \none & \none & c_{3}  \\
   \none & c_{0} & c_{1} &  c_{2} \\
 c_{-2}    & c_{-1} & c_{0} & c_{1}  \\ 
 c_{-3}    & c_{-2} & c_{-1} & c_{0} 
  \end{ytableau}
  =
  \left|
  \begin{matrix}
\ytableausetup{boxsize=normal}
  \begin{ytableau}
   \none & \none & \none & c_{3}  \\
   \none &  \none &  \none &  c_{2} \\
  \none  &  \none &  \none & c_{1}  \\ 
 c_{-3}    & c_{-2} & c_{-1} & c_{0} 
  \end{ytableau}

& 

\ytableausetup{boxsize=normal}
  \begin{ytableau}
 \none & \none & c_{3}  \\
  \none &  \none &  c_{2} \\
  \none &  \none & c_{1}  \\ 
 c_{-2} & c_{-1} & c_{0} 
  \end{ytableau}

&

\ytableausetup{boxsize=normal}
  \begin{ytableau}
 c_{3}  \\
  c_{2} \\
 c_{1}  \\ 
 c_{0} 
  \end{ytableau}

\\
\\

\ytableausetup{boxsize=normal}
  \begin{ytableau}
  \none  &  \none &  \none & c_{1}  \\ 
 c_{-3}    & c_{-2} & c_{-1} & c_{0} 
  \end{ytableau}

&

\ytableausetup{boxsize=normal}
  \begin{ytableau}
  \none &  \none & c_{1}  \\ 
 c_{-2} & c_{-1} & c_{0} 
  \end{ytableau}

&

\ytableausetup{boxsize=normal}
  \begin{ytableau}
c_{1}  \\ 
 c_{0} 
  \end{ytableau}

\\
\\

\ytableausetup{boxsize=normal}
  \begin{ytableau}
 c_{-3}    & c_{-2} & c_{-1} & c_{0} 
  \end{ytableau}

&

\ytableausetup{boxsize=normal}
  \begin{ytableau}
 c_{-2} & c_{-1} & c_{0} 
  \end{ytableau}

&

\ytableausetup{boxsize=normal}
  \begin{ytableau}
 c_{0} 
  \end{ytableau}
  \end{matrix}
  \right|
  .
$$
\end{example}

\begin{example}
Put $
 \begin{ytableau}
{\pmb \alpha}
  \end{ytableau}
=
 \begin{ytableau}
\alpha_1 & \cdots & \alpha_n 
  \end{ytableau}
$
and
$
 \begin{ytableau}
{\pmb \beta}
  \end{ytableau}
=
 \begin{ytableau}
\beta_m \\ \vdots \\ \beta_1 
  \end{ytableau}
$ with 
$
 \begin{ytableau}
{\pmb \alpha}^{\#}
  \end{ytableau}
=
 \begin{ytableau}
\alpha_1 & \cdots & \alpha_n 
  \end{ytableau}^{\#}
$
and
$
 \begin{ytableau}
{\pmb \beta}^{\#}
  \end{ytableau}
=
 \begin{ytableau}
\beta_m \\ \vdots \\ \beta_1 
  \end{ytableau}^{\#}
$.
For $\lambda=(n+3, 3, 3, \underbrace{1, \cdots, 1}_{m})$ and
$$
{\pmb \gamma}=
\ytableausetup{boxsize=normal}
  \begin{ytableau}
 c_{0} & c_{1} &  c_{2} & {\pmb \alpha} \\
 c_{-1} & c_{0} & c_{1}  \\ 
 c_{-2} & c_{-1} & c_{0} \\
 {\pmb \beta}
  \end{ytableau},
$$
Theorem \ref{Giambelli2} shows
\begin{equation}\label{Exantipode}
\ytableausetup{boxsize=normal}
  \begin{ytableau}
   \none & \none & \none & {\pmb \alpha}^{\#}  \\
   \none & c_{0} & c_{1} &  c_{2} \\
  \none  & c_{-1} & c_{0} & c_{1}  \\ 
 {\pmb \beta}^{\#}   & c_{-2} & c_{-1} & c_{0} 
  \end{ytableau}
  =
  \left|
  \begin{matrix}
\ytableausetup{boxsize=normal}
  \begin{ytableau}
   \none & \none & \none & {\pmb \alpha}^{\#}  \\
   \none& \none& \none&  c_{2} \\
   \none & \none & \none &  c_{1}  \\ 
  {\pmb \beta}^{\#}  & c_{-2}   &  c_{-1}  & c_{0}
  \end{ytableau}

&

\ytableausetup{boxsize=normal}
  \begin{ytableau}
 \none & {\pmb \alpha}^{\#}  \\
 \none&  c_{2} \\
 \none &  c_{1}  \\ 
  c_{-1}  & c_{0}
  \end{ytableau}

&

\ytableausetup{boxsize=normal}
  \begin{ytableau}
 {\pmb \alpha}^{\#}  \\
  c_{2} \\
  c_{1}  \\ 
 c_{0}
  \end{ytableau}

\\
\\

\ytableausetup{boxsize=normal}
  \begin{ytableau}
   \none & \none & \none &  c_{1}  \\ 
  {\pmb \beta}^{\#}   & c_{-2}   &  c_{-1}  & c_{0}
  \end{ytableau}

&

\ytableausetup{boxsize=normal}
  \begin{ytableau}
 \none &  c_{1}  \\ 
  c_{-1}  & c_{0}
  \end{ytableau}

&

\ytableausetup{boxsize=normal}
  \begin{ytableau}
c_{1}  \\ 
 c_{0}
  \end{ytableau}

\\
\\

\ytableausetup{boxsize=normal}
  \begin{ytableau}
  {\pmb \beta}^{\#}   & c_{-2}   &  c_{-1}  & c_{0}
  \end{ytableau}

&

\ytableausetup{boxsize=normal}
  \begin{ytableau}
  c_{-1}  & c_{0}
  \end{ytableau}

&

\ytableausetup{boxsize=normal}
  \begin{ytableau}
 c_{0}
  \end{ytableau}

\end{matrix}
\right|
\end{equation}
\end{example}

\begin{remark}
We might extend Theorem \ref{Giambelli2} to that of laced type. But we do not discuss it here.
\end{remark}

\begin{remark}\label{R3-6}
In the previous paper \cite[Theorem 4.1]{MatNak21}, we gave an expression of the Schur multiple zeta-functions of anti-hook type in terms of modified zeta-functions of root systems.
We only describe the results here and will leave the detailed definition of each function to the reader.
Let $\theta=
\underbrace{
\ytableausetup{boxsize=normal}
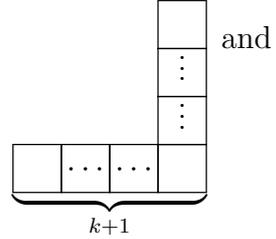
  \begin{ytableau}
   \none & \none & \none &   \\
   \none &  \none &  \none & \vdots \\
  \none  & \none  & \none  &   \vdots \\ 
     & \cdots & \cdots &  
  \end{ytableau}
}_{k+1}
$
and
${\frak u}={\frak u}(k, \ell)=({\bf u}_1, {\bf u}_2, \cdots, {\bf u}_{k+\ell+1})$.
Then 
for $${\pmb s}\in W_{\theta}^{\circ}
=
\left\{{\pmb s}=(s_{ij})\in T(\theta,\mathbb{C})\, 
 \;| \;
 \Re(s_{ij})> 1 \;\text{for all}\; (i,j)\in \theta 
\right\},
$$ 
we have
\begin{eqnarray}\label{previous41}
 \zeta_{\theta}({\pmb s})  = 
 \zeta_{k+\ell+1, {k+1}}^{\bullet}({\frak u}, A_{k+\ell+1})
+
\sum_{\mu=0}^{k-1}\sum_{\nu=0}^{\ell-1}(-1)^{\mu+\nu}
Z_{\mu,\nu}({\pmb s};(k,\ell)),
\label{thm2}
\end{eqnarray}
where $\zeta_{k+\ell+1, {k+1}}^{\bullet}({\frak u}, A_{k+\ell+1})$ is a modified zeta-function of root systems of type $A$
and
$Z_{\mu,\nu}$ is a double series which can be regarded as an analogue of ``Weyl group multiple Dirichlet series" in the sense of Bump \cite{Bump}.
Then the $(i, j)$-component of the matrix on the right-hand side of \eqref{Giambelli2-formula}, and hence also the left-hand side of
 \eqref{Giambelli2-formula}, are 
written in terms of \eqref{previous41}.
\end{remark}

\section{Some new identities among Schur multiple zeta-functions}\label{sec-4}


As mentioned in Section \ref{sec-1}, in this section we consider the
tableaux of winged type whose middle part is of rectangle shape (see the figure)
$$
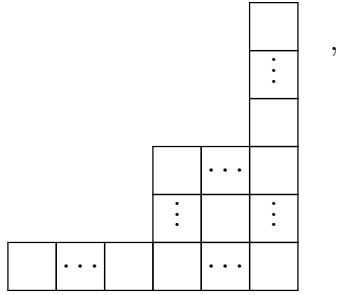
\ytableausetup{boxsize=normal}
  \begin{ytableau}
   \none & \none & \none & \none & \none &  \\
   \none & \none & \none & \none & \none & \vdots \\
   \none & \none & \none & \none & \none &  \\
   \none & \none & \none &  & \cdots &  \\
   \none & \none & \none & \vdots &  & \vdots \\
     & \cdots &  &  & \cdots &
  \end{ytableau}\quad,
$$
and we deduce new identities among Schur multiple zeta-functions. This shape is of winged
type and so we can apply Theorem \ref{Giambelli1} in Section \ref{sec-2}, while this shape
can be also recognized as a special case of skew type discussed in Section \ref{sec-3}, hence
we may apply Theorem \ref{Giambelli2}.    
Equalizing the two expressions obtained from those two theorems, we find that the following
theorem holds.

\begin{thm}\label{Giambelli4} 
Let  $\lambda$ be a partition, and write its  
Frobenius notation as $\lambda=(p_1, \cdots , p_N | q_1, \cdots, q_N)$.
Assume the Young diagram of shape $\lambda$ is the form of rectangle shape and
$\widetilde{\lambda}$ is of winged type whose middle part is $\lambda$.
Then the Young diagram of shape $\widetilde{\lambda}^{\#}$ is of standard (non-skewed) type. 
Let $\boldsymbol{\gamma}\in T(\widetilde{\lambda}^{\#}, {\bf N})$ and put 
$\widetilde{\pmb s}=\boldsymbol{\gamma}^{\#}\in T(\widetilde{\lambda}, {\bf N})$.
Then
using the notation $\widetilde{\zeta}_{ij}$ as in Theorem \ref{Giambelli1} and ${\gamma}_{ij}$ as in Theorem \ref{Giambelli2}, we have
\begin{equation}\label{expij-1}
\det(\widetilde{\zeta}_{i,j})_{1 \leq i,j \leq N}=
\det(\zeta_{(p_i,1^{q_j})^{\#}}(\gamma_{ij}^{F\#}))_{1\leq i,j\leq N}.
\end{equation}
\end{thm}

The simplest example of this theorem was already given in Section \ref{sec-1}
(see Example \ref{example-1}).
Here we add some more examples.

\begin{example}
If $
\ytableausetup{boxsize=normal}
  \begin{ytableau}
 {\pmb s}^{*}
  \end{ytableau}
=\begin{ytableau}
 {\pmb \alpha}^{\#}
  \end{ytableau}
$ and
$
\ytableausetup{boxsize=normal}
  \begin{ytableau}
 {\pmb s}_{*}
  \end{ytableau}
=\begin{ytableau}
 {\pmb \beta}^{\#}
  \end{ytableau}
$ in \eqref{Ex210}, by comparing this with \eqref{Exantipode}, we get the following identity:
$$
  \left|
  \begin{matrix}
\ytableausetup{boxsize=normal}
  \begin{ytableau}
   \none & \none & \none & {\pmb \alpha}^{\#}  \\
   \none & c_{0}  &  c_{1} &  c_{2} \\
  \none  &  c_{-1}  \\ 
  {\pmb \beta}^{\#}   & c_{-2}   
  \end{ytableau}

&

\ytableausetup{boxsize=normal}
  \begin{ytableau}
   \none & \none & \none & {\pmb \alpha}^{\#}  \\
   \none & c_{0}  &  c_{1} &  c_{2} \\
  \none  &  c_{-1}  
   \end{ytableau}

&

\ytableausetup{boxsize=normal}
  \begin{ytableau}
   \none & \none & \none & {\pmb \alpha}^{\#}  \\
   \none & c_{0}  &  c_{1} &  c_{2} 
   \end{ytableau}

\\
\\

\ytableausetup{boxsize=normal}
  \begin{ytableau}
   \none & c_{0}  &  c_{1} \\
  \none  &  c_{-1}  \\ 
  {\pmb \beta}^{\#}   & c_{-2}   
  \end{ytableau}

&

\ytableausetup{boxsize=normal}
  \begin{ytableau}
   \none & c_{0}  &  c_{1} \\
  \none  &  c_{-1}  
   \end{ytableau}

&

\ytableausetup{boxsize=normal}
  \begin{ytableau}
   \none & c_{0}  &  c_{1} 
   \end{ytableau}

\\
\\

\ytableausetup{boxsize=normal}
  \begin{ytableau}
   \none & c_{0}  \\
  \none  &  c_{-1}  \\ 
  {\pmb \beta}^{\#}   & c_{-2}   
  \end{ytableau}

&

\ytableausetup{boxsize=normal}
  \begin{ytableau}
   \none & c_{0}  \\
  \none  &  c_{-1}  
   \end{ytableau}

&

\ytableausetup{boxsize=normal}
  \begin{ytableau}
   \none & c_{0} 
   \end{ytableau}

\end{matrix}
\right|
  =
  \left|
  \begin{matrix}
\ytableausetup{boxsize=normal}
  \begin{ytableau}
   \none & \none & \none & {\pmb \alpha}^{\#}  \\
   \none& \none& \none&  c_{2} \\
   \none & \none & \none &  c_{1}  \\ 
  {\pmb \beta}^{\#}   & c_{-2}   &  c_{-1}  & c_{0}
  \end{ytableau}

&

\ytableausetup{boxsize=normal}
  \begin{ytableau}
 \none & {\pmb \alpha}^{\#}  \\
 \none&  c_{2} \\
 \none &  c_{1}  \\ 
  c_{-1}  & c_{0}
  \end{ytableau}

&

\ytableausetup{boxsize=normal}
  \begin{ytableau}
 {\pmb \alpha}^{\#}  \\
  c_{2} \\
  c_{1}  \\ 
 c_{0}
  \end{ytableau}

\\
\\

\ytableausetup{boxsize=normal}
  \begin{ytableau}
   \none & \none & \none &  c_{1}  \\ 
  {\pmb \beta}^{\#}   & c_{-2}   &  c_{-1}  & c_{0}
  \end{ytableau}

&

\ytableausetup{boxsize=normal}
  \begin{ytableau}
 \none &  c_{1}  \\ 
  c_{-1}  & c_{0}
  \end{ytableau}

&

\ytableausetup{boxsize=normal}
  \begin{ytableau}
c_{1}  \\ 
 c_{0}
  \end{ytableau}

\\
\\

\ytableausetup{boxsize=normal}
  \begin{ytableau}
  {\pmb \beta}^{\#}   & c_{-2}   &  c_{-1}  & c_{0}
  \end{ytableau}

&

\ytableausetup{boxsize=normal}
  \begin{ytableau}
  c_{-1}  & c_{0}
  \end{ytableau}

&

\ytableausetup{boxsize=normal}
  \begin{ytableau}
 c_{0}
  \end{ytableau}

\end{matrix}
\right|
$$
\end{example}
\begin{example}
Applying Theorem \ref{Giambelli1} to ${\pmb \gamma}=\ytableausetup{boxsize=normal}
  \begin{ytableau}
   \none & \none & \none & {\pmb \alpha}  \\
   \none & c_{0} & c_{1} &  c_{2} \\
  \none  & c_{-1} & c_{0} & c_{1}  \\ 
\none   & c_{-2} & c_{-1} & c_{0} \\
 {\pmb \beta}    & c_{-3} & c_{-2} & c_{-1} 
  \end{ytableau}
$ and Theorem \ref{Giambelli2} to ${\pmb \gamma}^{\#}$ gives following:
$$
  \left|
  \begin{matrix}
\ytableausetup{boxsize=normal}
  \begin{ytableau}
   \none & \none & \none & {\pmb \alpha}  \\
   \none & c_{0}  &  c_{1} &  c_{2} \\
  \none  &  c_{-1}  \\ 
\none  & c_{-2}\\   
  {\pmb \beta}   & c_{-3}   
  \end{ytableau}

&

\ytableausetup{boxsize=normal}
  \begin{ytableau}
   \none & \none & \none & {\pmb \alpha}  \\
   \none & c_{0}  &  c_{1} &  c_{2} \\
  \none  &  c_{-1}  \\
    \none  &  c_{-2}  
   \end{ytableau}

&

\ytableausetup{boxsize=normal}
  \begin{ytableau}
   \none & \none & \none & {\pmb \alpha}  \\
   \none & c_{0}  &  c_{1} &  c_{2}\\
     \none  &  c_{-1}   
   \end{ytableau}

\\
\\

\ytableausetup{boxsize=normal}
  \begin{ytableau}
   \none & c_{0}  &  c_{1} \\
  \none  &  c_{-1}  \\ 
\none  & c_{-2}   \\
    {\pmb \beta}   & c_{-3}   
  \end{ytableau}

&

\ytableausetup{boxsize=normal}
  \begin{ytableau}
   \none & c_{0}  &  c_{1} \\
  \none  &  c_{-1}  \\
    \none  &  c_{-2}  
   \end{ytableau}

&

\ytableausetup{boxsize=normal}
  \begin{ytableau}
   \none & c_{0}  &  c_{1}\\
     \none  &  c_{-1}   
   \end{ytableau}

\\
\\

\ytableausetup{boxsize=normal}
  \begin{ytableau}
   \none & c_{0}  \\
  \none  &  c_{-1}  \\ 
\none  & c_{-2}\\   
  {\pmb \beta}   & c_{-3}   
    \end{ytableau}

&

\ytableausetup{boxsize=normal}
  \begin{ytableau}
   \none & c_{0}  \\
  \none  &  c_{-1}  \\
\none  & c_{-2}   
   \end{ytableau}

&

\ytableausetup{boxsize=normal}
  \begin{ytableau}
   \none & c_{0} \\
      \none & c_{-1} 

   \end{ytableau}

\end{matrix}
\right|
  =
  \left|
  \begin{matrix}
\ytableausetup{boxsize=normal}
  \begin{ytableau}
   \none & \none & \none & {\pmb \alpha}  \\
   \none& \none& \none&  c_{2} \\
   \none & \none & \none &  c_{1}  \\ 
   \none & \none & \none &  c_{0}  \\ 
  {\pmb \beta}   & c_{-3}   &  c_{-2}  & c_{-1}
  \end{ytableau}

&

\ytableausetup{boxsize=normal}
  \begin{ytableau}
 \none & {\pmb \alpha}  \\
 \none&  c_{2} \\
 \none &  c_{1}  \\ 
 \none &  c_{0}  \\ 
   c_{-2}  & c_{-1}
  \end{ytableau}

&

\ytableausetup{boxsize=normal}
  \begin{ytableau}
 {\pmb \alpha}  \\
  c_{2} \\
  c_{1}  \\ 
 c_{0}\\
 c_{-1}
  \end{ytableau}

\\
\\

\ytableausetup{boxsize=normal}
  \begin{ytableau}
   \none & \none & \none &  c_{1}  \\ 
   \none & \none & \none &  c_{0}  \\ 
     {\pmb \beta}   & c_{-3}   &  c_{-2}  & c_{-1}
  \end{ytableau}

&

\ytableausetup{boxsize=normal}
  \begin{ytableau}
 \none &  c_{1}  \\ 
 \none &  c_{0}  \\ 
   c_{-2}  & c_{-1}
  \end{ytableau}

&

\ytableausetup{boxsize=normal}
  \begin{ytableau}
c_{1}  \\ 
 c_{0}\\
  c_{-1}
  \end{ytableau}

\\
\\

\ytableausetup{boxsize=normal}
  \begin{ytableau}
  \none & \none & \none & c_0\\
  {\pmb \beta}   & c_{-3}   &  c_{-2}  & c_{-1}
  \end{ytableau}

&

\ytableausetup{boxsize=normal}
  \begin{ytableau}
  \none & c_0 \\
  c_{-2}  & c_{-1}
  \end{ytableau}

&

\ytableausetup{boxsize=normal}
  \begin{ytableau}
 c_{0}\\
 c_{-1}
  \end{ytableau}

\end{matrix}
\right| \quad.
$$
\end{example}

\section{Proof of Theorem \ref{Giambelli1}}\label{sec-5}

Let $\lambda$ be a partition.   We first recall the Frobenius notation of $\lambda$
introduced in Section \ref{sec-1}.    In terms of the corresponding Young diagram, 
$p_i$ is the number of boxes from $(i,i+1)$-th to $(i,\lambda_i)$-th, 
and $D(p_i)$ the set of all those boxes.
Similarly, $q_i$ the
number of boxes from $(i+1,i)$-th to $(\lambda_i',i)$-th and $D(q_i)$ the set of all
those boxes.
Putting $p=p(\lambda)=(p_1,\ldots,p_N)$ and
$q=q(\lambda)=(q_1,\ldots,q_N)$, we may write $\lambda=(p|q)$.
Further we put
\begin{align*}
D(p)=\bigcup_{1\leq i\leq N}D(p_i),\quad D(q)=\bigcup_{1\leq i\leq N} D(q_i),
\quad B=\bigcup_{1\leq i\leq N}B_i,
\end{align*}
where by $B_i$ we denote the $i$-th box on the main diagonal.

For example, in the case of the Young diagram of shape $\lambda=(3,3,2,1)$
$$
  \ydiagram[]{3,3,2,1}\; ,
$$
the right of the northeast boundary of the main diagonal is
$\alpha(\lambda)=(2,1)$, depicted as 
$
\ydiagram{2,1+1}
$\; ,
while the lower of the southwest boundary of the main diagonal is $\beta(\lambda)=(3,1)$, 
which is depicted as
$
\ydiagram{1,2,1}
$\; , and so the Frobenius notation of this example is $((2,1)|(3,1))$.

\begin{remark}
A slightly different notation is used in \cite{EgeRem88}, where $q_i$ is
defined as $q_i=\lambda_i'-i+1$; that is, the number of boxes from $(i,i)$-th to
$(\lambda_i',i)$-th.
\end{remark}

\subsection{Single-laced type}
We first aim to prove Theorem \ref{Giambelli1} in the single-laced case, that is $\lambda^*=\lambda_*=\emptyset$ in $\widetilde{\lambda}$. 
Let ${\pmb s}\in W_{\lambda}^{\rm{diag}}$.    Our aim is to show
\begin{align}\label{non-winged}
\zeta_{\lambda}({\pmb s})=\det(\zeta_{i,j}({\pmb s}_{i,j}^F))_{1\leq i,j\leq N}.
\end{align}

Let ${\frak S}_N$ be the $N$-th symmetric group. 
The right-hand side of \eqref{non-winged} is
\begin{align}\label{firststep}
&\det(\zeta_{i,j}({\pmb s}_{i,j}^F))_{1\leq i,j\leq N}
=\sum_{\sigma\in {\frak S}_N}{\rm sgn}(\sigma)\prod_{k=1}^N \zeta_{k, \sigma(k)}
({\pmb s}_{k, \sigma(k)}^F)\\
&\quad =\sum_{\sigma\in {\frak S}_N}{\rm sgn}(\sigma)\prod_{k=1}^N \sum_{M_k(\sigma)\in 
{\rm{SSYT}}(p_k | q_{\sigma(k)})}w({\pmb s}_{k,\sigma(k)}^F, M_k(\sigma)),\notag
\end{align}
by the definition \eqref{1-1}.

For a given permutation $\sigma\in {\frak S}_N$ and a Young tableau 
$T=(t_{i,j})$ of shape $\lambda$, 
we call $T$ a $\sigma$-tableau if it satisfies the following conditions:

(I) The entries of $T$ are weakly increasing along the rows in $D(p)$.

(II) The entries of $T$ are strictly increasing down the columns in $D(q)\cup B$.

(III) $t_{\sigma(i), \sigma(i)}\leq t_{i, i+1}$ for $i=1, 2, \ldots, d$ whenever $i+1\leq \lambda_i$. 

\noindent
Let $T$ be a $\sigma$-tableau, and define
\begin{align}\label{sigma-tableau}
T_k(\sigma)=
  \begin{array}{|c|c|c|c|}
 \hline
 t_{\sigma(k),\sigma(k)} & t_{k,k+1} & \cdots & t_{k,k+p_k}\\
 \hline
 t_{\sigma(k),\sigma(k)+1}\\
 \cline{0-0}
 \vdots\\
 \cline{0-0}
 t_{\sigma(k),\sigma(k)+q_{\sigma(k)}}\\
 \cline{0-0}
  \end{array}
  \; .
\end{align}
We sometimes write $T_k(\sigma)=T^F_{k, \sigma(k)}$.

%

\begin{lem}\label{sigmatab}
Let $T=(t_{ij})$ be a $\sigma$-tableau of shape $\lambda=(p | q)$. 
Then 
$T_k(\sigma)\in {\rm SSYT}(p_k | q_{\sigma(k)})$ for $1\leq k\leq N$,
and
$$
w({\pmb s}, T)=\prod_{k=1}^N w({\pmb s}_{k,\sigma(k)}^F, T_k(\sigma)).
$$
\end{lem}
\begin{proof}
By the condition (I) of  $\sigma$-tableau, $t_{k,k+1} \leq \cdots \leq  t_{k,k+\alpha_k}$, 
and by the condition (II),
$ t_{\sigma(k),\sigma(k)}<
 t_{\sigma(k),\sigma(k)+1}<\cdots
 t_{\sigma(k),\sigma(k)+\beta_{\sigma(k)}}
$.
Further
by the condition (III) of $\sigma$-tableau, $t_{\sigma(k), \sigma(k)}\leq t_{k,k+1}$. 
So $T_k(\sigma)$ is SSYT.

By the definition of $w({\pmb s}, T)$, we have
\begin{align}
&w({\pmb s}, T)=\prod_{(i,j)\in {\lambda}}\frac{1}{t_{i,j}^{s_{i,j}}}
=
\prod_{(i,j)\in D(p)}\frac{1}{t_{ij}^{s_{i,j}}}\prod_{(i,j)\in D(q)\cup B}
\frac{1}{t_{ij}^{s_{i,j}}}\notag\\
&=\prod_{(i,j)\in D(p_1)}\frac{1}{t_{ij}^{s_{ij}}}\cdots
\prod_{(i,j)\in D(p_N)}\frac{1}{t_{ij}^{s_{ij}}}
\prod_{(i,j)\in D(q_{\sigma(1)})\cup B_{\sigma(1)}}\frac{1}{t_{i,j}^{s_{i,j}}}
\cdots
\prod_{(i,j)\in D(q_{\sigma(N)})\cup B_{\sigma(N)}}\frac{1}{t_{i,j}^{s_{i,j}}}\notag\\
&=\left(\prod_{(i,j)\in D(p_1)}\frac{1}{t_{i,j}^{s_{i,j}}}
\prod_{(i,j)\in D(q_{\sigma(1)})\cup B_{\sigma(1)}}\frac{1}{t_{i,j}^{s_{i,j}}}\right)
\cdots
\left(\prod_{(i,j)\in D(p_N)}\frac{1}{t_{i,j}^{s_{i,j}}}
\prod_{(i,j)\in D(q_{\sigma(N)})\cup B_{\sigma(N)}}\frac{1}{t_{i,j}^{s_{i,j}}}\right)\notag\\
&=\prod_{k=1}^N w({\pmb s}_{k, \sigma(k)}^F, T_k(\sigma)).\label{keisan}
\end{align}
\end{proof}


Conversely, for any given 
$T_k(\sigma)\in {\rm SSYT}(p_k | q_{\sigma(k)})$ ($1\leq k\leq N$),
we may construct a $\sigma$-tableau $T$ by combining those $T_k(\sigma)$ as in the
above proof.    
Therefore, putting
$${\mathbb X}:=\{(\sigma, T) | \sigma\in {\frak S}_N, T : \sigma \mbox{-tableau of shape} \; \lambda\},
$$
from \eqref{firststep} and Lemma \ref{sigmatab} we find

\begin{align*}
\det(\zeta_{i,j}({\pmb s}_{i,j}^F))_{1\leq i,j\leq N}
=\sum_{\sigma\in {\frak S}_d}{\rm sgn}(\sigma) \sum_{T : \sigma{\rm -tableau}}w({\pmb s}, T)
=\sum_{(\sigma, T)\in {\mathbb X}}{\rm sgn}(\sigma) w({\pmb s}, T).
\end{align*}
We denote the right-hand side of the above by $W_{\mathbb X}$.

Let $T\in {\rm SSYT}(\lambda)$.   Then $T$ is obviously an $\mathbbm{1}$-tableau, where
$\mathbbm{1}$ is the identity permutation.
However $T$ is never a $\sigma$-tableau for any $\sigma\neq \mathbbm{1}$.
In fact, if $T$ is a $\sigma$-tableau, then 
$t_{\sigma(i),\sigma(i)}\leq t_{i,i+1}$ for any $i$.
If $\sigma\neq\mathbbm{1}$, then there exists an $i$ for which $j=\sigma(i)>i$, so
$t_{j,j}\leq t_{i,i+1}$.   However, since $T$ is SSYT, it should be
$t_{i,i+1}\leq t_{j-1,j}<t_{j,j}$, which is a contradiction.

We devide
$\mathbb{X}={\mathbb X}^{\dagger}\cup {\mathbb X}^{\dagger\dagger}$, 
where ${\mathbb X}^{\dagger}:=\{(\mathbbm{1}, T) | T\in {\rm SSYT}({\lambda})\}$
and ${\mathbb X}^{\dagger\dagger}={\mathbb X}\setminus{\mathbb X}^{\dagger}$.
Then $T$ is not SSYT for any element 
$(\sigma,T)\in {\mathbb X}^{\dagger\dagger}$.
Since
\begin{align*}
W_{{\mathbb X}^{\dagger}}:&=
\displaystyle{\sum_{(\mathbbm{1}, T)\in {\mathbb X}^{\dagger}}
{\rm sgn}(\mathbbm{1})w({\pmb s}, T)}
=\sum_{T\in {\rm SSYT}({\lambda})}w({\pmb s}, T)=\zeta_{(\alpha | \beta)},
\end{align*}
to prove \eqref{non-winged}
it is sufficient to show that $W_{\mathbb X}=W_{{\mathbb X}^{\dagger}}$. In other words,
it is enough to show that 
\begin{align}\label{daggerdagger}
W_{{\mathbb X}^{\dagger\dagger}}:=
\sum_{(\sigma, T)\in {\mathbb X}^{\dagger\dagger}}{\rm sgn}(\sigma)w({\pmb s}, T)=0.
\end{align}
Now we proceed to show 
\begin{lem}\label{cancellation}
For any $(\sigma,T)\in {\mathbb X}^{\dagger\dagger}$, 
there exists another element $(\sigma',T')\in {\mathbb X}^{\dagger\dagger}$ such that
\begin{align}\label{daggerzero}
{\rm sgn}(\sigma)w({\pmb s},T)+{\rm sgn}(\sigma')w({\pmb s},T')=0.
\end{align}
\end{lem}

\begin{proof}
Let $(\sigma,T)\in {\mathbb X}^{\dagger\dagger}$.    Then $T \notin {\rm SSYT}(\lambda)$.
The situation can be divided into the following three cases.

(i) $T$ does not satisfy the SSYT condition in the region $D(p)$;

(ii) $T$ satisfies the SSYT condition in $D(p)$, but does not satisfy it in the
region $D(q)\cup B$;

(iii) $T$ satisfies the SSYT condition both in $D(p)$ and in $D(q)\cup B$, but does not
satisfy it on the north-east boundary of the main diagonal. 

We treat these three cases separately.

{\it Case (i)}.
Assume that the SSYT condition is violated in $D(p)$, 
between the $i$-th row and the $(i+1)$-th
row.    We can find a $k (\geq i+2)$ for which $t_{i,k}\geq t_{i+1,k}$.
We may assume that $k$ is the largest of such.
Let $\sigma'=\sigma\circ (i,i+1)$, and 
define $T'=(t_{ij}')$ by changing the places of the elements $t_{i,h}$ and $t_{i+1,h+1}$ of $T$
($i+1\leq h\leq k-1$), while all the other elements remain the same as $T$.

We show that $T'$ is a $\sigma'$-tableau.
In fact, the $i$-th row of $T'$ is weakly increasing because $t_{i+1,k}\leq t_{i,k}$,
and the $(i+1)$-th row of $T'$ is also weakly increasing because
$t_{i,k-1}\leq t_{i,k}\leq t_{i,k+1}< t_{i+1,k+1}$ (where the last inequality comes from
the maximality of $k$), so the condition (I) is valid.
The condition (II) is obviously valid because $T'$ in $D(q)\cup B$ is 
exactly the same as $T$.

Next we check the condition (III).   Since $\sigma'=\sigma\circ (i,i+1)$, we find
\begin{align*}
&t'_{\sigma'(i),\sigma'(i)}=t'_{\sigma(i+1),\sigma(i+1)}
=t_{\sigma(i+1),\sigma(i+1)}\leq t_{i+1,i+2}=t'_{i,i+1},\\
&t'_{\sigma'(i+1),\sigma'(i+1)}=t'_{\sigma(i),\sigma(i)}
=t_{\sigma(i),\sigma(i)}\leq t_{i,i+1}=t'_{i+1,i+2}
\end{align*}
and, for any $\ell\neq i,i+1$, 
$$
t'_{\sigma'(\ell),\sigma'(\ell)}=t'_{\sigma(\ell),\sigma(\ell)}
=t_{\sigma(\ell),\sigma(\ell)}\leq t_{\ell,\ell+1}=t'_{\ell,\ell+1},
$$
hence (III) is valid.   Therefore $T'$ is a $\sigma'$-tableau.

Since ${\pmb s}\in W_{\lambda}^{\rm diag}$, we see that
$w({\pmb s},T)=w({\pmb s},T')$.    
Noting ${\rm sgn}(\sigma')=-{\rm sgn}(\sigma)$, we obtain \eqref{daggerzero}
in this case.

{\it Case (ii)}.
Assume that the SSYT condition is valid in $D(p)$, but is
violated in $D(q)\cup B$, between the $j$-th column and the $(j+1)$-th
column.    We can find a $k (\geq j+1)$ for which $t_{k,j}> t_{k,j+1}$.
We may assume that $k$ is the largest of such.
Let $\sigma'=(j,j+1)\circ\sigma$, and this time
define $T'=(t_{ij}')$ by changing the places of the elements $t_{h,j}$ and $t_{h+1,j+1}$ 
of $T$
($j\leq h\leq k-1$), while all the other elements remain the same as $T$.

Again we can show that $T'$ is a $\sigma'$-tableau.
The condition (I) is obvious, while the condition (II) can be shown by
$t_{k,j+1}<t_{k,j}$ and
$t_{k-1,j}<t_{k,j}<t_{k+1,j}\leq t_{k+1,j+1}$.
Let us check the condition (III).   For any $k$, we have
\begin{align*}
&\sigma'(k)=
   \begin{cases}
   \sigma(k) & {\rm if} \quad \sigma(k)\neq j,j+1, \\
   j+1 & {\rm if}  \quad \sigma(k)=j, \\
   j & {\rm if}\quad  \sigma(k)=j+1.
   \end{cases}
\end{align*}
Therefore we find
\begin{align*}
t'_{\sigma'(k),\sigma'(k)}=
\begin{cases}
t'_{\sigma(k),\sigma(k)}=t_{\sigma(k),\sigma(k)}
\leq t_{k,k+1}=t'_{k,k+1} & {\rm if} \quad \sigma(k)\neq j,j+1, \\
t'_{j+1,j+1}=t_{j,j}=t_{\sigma(k),\sigma(k)}\leq t_{k,k+1} =t'_{k,k+1}
& {\rm if}  \quad \sigma(k)=j, \\
t'_{j,j}=t_{j+1,j+1}=t_{\sigma(k),\sigma(k)}
\leq t_{k,k+1}=t'_{k,k+1} & {\rm if}\quad  \sigma(k)=j+1,
\end{cases}
\end{align*}
which implies the condition (III).
Therefore $T'$ is a $\sigma'$-tableau.
Since ${\pmb s}\in W_{\lambda}^{\rm diag}$ and
${\rm sgn}(\sigma')=-{\rm sgn}(\sigma)$, \eqref{daggerzero} follows in this case.

{\it Case (iii)}.
Assume that the SSYT condition is valid in $D(p)$, also in $D(q)\cup B$, but is
violated on the north-east boundary of the main diagonal.
Let $k$ be the least such that the violation happens at $t_{k,k+1}$; that is, 
the first violation is either
(A) $t_{k,k}>t_{k,k+1}$, or (B) $t_{k+1,k+1}\leq t_{k,k+1}$.

First we claim that $\sigma(h)=h$ for all $h\leq k-1$ (when $k\geq 2$).
Put $r=\sigma(1)$.    Then $t_{r,r}=t_{\sigma(1),\sigma(1)}\leq t_{1,2}$.
But if $r\geq 2$, then $t_{r,r}\geq t_{2,2}> t_{1,2}$ (because of the SSYT condition in 
$D(q)\cup B$ and the minimality of $k$), and hence $r=1$.
When $k=2$, we are done.   Let $k\geq 3$, and put $s=\sigma(2)$.
We have $t_{s,s}=t_{\sigma(2),\sigma(2)}\leq t_{2,3}$, while if $s\geq 3$ then
$t_{s,s}\geq t_{3,3}> t_{2,3}$, so $s=2$.
Repeating this procedure, we obtain the claim.

From the above claim we see that $\sigma(k)\geq k$.   Therefore
$t_{k,k}\leq t_{\sigma(k),\sigma(k)}\leq t_{k,k+1}$, so the case (A) does not happen.

Consider the case (B).
Define $i, j$ by $\sigma(i)=k$ and $\sigma(j)=k+1$.   By the above claim we have
$i,j\geq k$.   Put $\sigma'=\sigma\circ (i,j)$.   Then
\begin{align}\label{case3B-1}
t_{k+1,k+1}\leq t_{k,k+1}\leq t_{i,i+1}
\end{align}
because of (B) and the SSYT condition in $D(p)$.   Also,
\begin{align}\label{case3B-2}
t_{k,k}\leq t_{k+1,k+1}=t_{\sigma(j),\sigma(j)}\leq t_{j,j+1}
\end{align}
because of the SSYT condition in $D(q)\cup B$.

In this case we choose $T'=T$, and prove that $T$ is also a $\sigma'$-tableau.   
The conditions (I) and (II) are obvious. 
By \eqref{case3B-1}, \eqref{case3B-2} we have
\begin{align*}
&t_{\sigma'(i),\sigma'(i)}=t_{\sigma(j),\sigma(j)}=t_{k+1,k+1}
    \leq t_{i,i+1},\\
& t_{\sigma'(j),\sigma'(j)}=t_{\sigma(i),\sigma(i)}=t_{k,k}
    \leq t_{j,j+1},
\end{align*}
and further, 
$$
t_{\sigma'(\ell),\sigma'(\ell)}=t_{\sigma(\ell),\sigma(\ell)}
\leq t_{\ell,\ell+1}
$$
for any $\ell\neq i,j$.
Therefore the condition (III) is valid, so $T$ is a $\sigma'$-tableau.
Since ${\rm sgn}(\sigma')=-{\rm sgn}(\sigma)$, we obtain \eqref{daggerzero} in this case.
\end{proof}

From Lemma \ref{cancellation} we find that $W_{{\mathbb X}^{\dagger\dagger}}=0$,
which is \eqref{daggerdagger}.
Therefore we now arrive at the assertion \eqref{non-winged}.


\subsection{Laced type}
Let 
\begin{align*}
\widetilde{T}_k(\sigma)=
\ytableausetup{boxsize=normal}
  \begin{ytableau}
   \none & \none & \none & \none & T^*  \\
   \none &  *(lightgray) & *(lightgray) T_k(\sigma)& *(lightgray) & *(lightgray) \\
  \none  & *(lightgray) & \none & \none & \none  \\ 
 T_*    & *(lightgray) &  \none & \none & \none 
  \end{ytableau},
\end{align*}
being the laced-Young tableau of shape $[\lambda_* | (p_k | q_{\sigma(k)}) | \lambda^*]$,
where colored cells are $T_k(\sigma)$ as in \eqref{sigma-tableau}, 
$T^*$ is the Young tableau of shape $\lambda^*$ and 
$T_*$ is that of shape $\lambda_*$.
Note that  
$T^*$ (resp. $T_*$) does not appear  when $k\not=1$ (resp. $\sigma(k)\not=1$).

The right-hand side of \eqref{expij-1} is
\begin{align}\label{firststep}
&
\det(\widetilde{\zeta}_{i,j})_{1 \leq i,j \leq N}
=\sum_{\sigma\in {\frak S}_N}{\rm sgn}(\sigma)\prod_{k=1}^N \widetilde{\zeta}_{k, \sigma(k)}
(\widetilde{\pmb s}_{k, \sigma(k)}^F)\\
&\quad =\sum_{\sigma\in {\frak S}_N}{\rm sgn}(\sigma)\prod_{k=1}^N \sum_{\widetilde{M}_k(\sigma)\in 
{\rm{SSYT}}(\widetilde{p_k | q_{\sigma(k)}})}w(\widetilde{{\pmb s}}_{k,\sigma(k)}^F, \widetilde{M}_k(\sigma)),\notag
\end{align}
by the definition \eqref{1-1}, where 
$$(\widetilde{p_k | q_{\sigma(k)}})=
\begin{cases}
[({p_k | q_{\sigma(k)}})] & k\not=1, \sigma(k)\not=1\\
[\emptyset | ({p_k | q_{\sigma(k)}}) | \lambda^* ]& k=1, \sigma(k)\not=1\\
[\lambda_* | ({p_k | q_{\sigma(k)}}) | \emptyset ] & k\not=1, \sigma(k)=1\\
[\lambda_* | ({p_k | q_{\sigma(k)}}) | \lambda^* ]& k=1, \sigma(k)=1
\end{cases}.$$


%

\begin{lem}\label{sigmatab}
Let $T=(t_{ij})$ be a $\sigma$-tableau of shape $\lambda=(p | q)$.
Let $T^*$ (resp. $T_*$) be a semi-standard tableau of shape $\lambda^*$ (resp. $\lambda_*$). 
Denote $X^{\mathrm{NE}}$ (resp. $X^{\mathrm{SW}}$) by the most north-east (resp. most south-west) element  in $X$ for $X=T$, $T_*$ or $T^*$.
Assume $T^{\mathrm{NE}}<T^{*, {\mathrm{SW}}}$ and $T_{*}^{\mathrm{NE}}\leq T^{{\mathrm{SW}}}$.
Then 
$\widetilde{T}_k(\sigma)\in {\rm SSYT}(\lambda_* |(p_k | q_{\sigma(k)}) | \lambda^*)$ for $1\leq k\leq N$,
and
$$
w(\widetilde{{\pmb s}}, \widetilde{T})=\prod_{k=1}^N w(\widetilde{{\pmb s}}_{k,\sigma(k)}^F, \widetilde{T}_k(\sigma)).
$$
\end{lem}
\begin{proof}
By the condition (I) of  $\sigma$-tableau, $t_{k,k+1} \leq \cdots \leq  t_{k,k+\alpha_k}$, 
and by the condition (II),
$ t_{\sigma(k),\sigma(k)}<
 t_{\sigma(k),\sigma(k)+1}<\cdots
 t_{\sigma(k),\sigma(k)+\beta_{\sigma(k)}}
$.
Further
by the condition (III) of $\sigma$-tableau, $t_{\sigma(k), \sigma(k)}\leq t_{k,k+1}$. 
So $T_k(\sigma)$ is SSYT.
From the assumptions for $T^*$ and $T_*$, and the conditions for the boundary, $\widetilde{T}_k(\sigma)$ is SSYT.

By the definition of $w(\widetilde{\pmb s}, \widetilde{T})$, we have
\begin{align}
&w(\widetilde{\pmb s}, \widetilde{T})=
\frac{1}{T^{\pmb s}(T^*)^{{\pmb s}^*}T_*^{{\pmb s}_*}}
=
\left(
\prod_{(i,j)\in {\lambda}}\frac{1}{t_{i,j}^{s_{i,j}}}\right)
\frac{1}{(T^*)^{{\pmb s}^*}T_*^{{\pmb s}_*}}.\label{keisan2}
\end{align}
The difference between \eqref{keisan} and  \eqref{keisan2} is 
that the term 
$\frac{1}{(T^*)^{{\pmb s}^*}T_*^{{\pmb s}_*}}$ is multiplied.
From
$$w(\widetilde{\pmb s}_{1, \sigma(1)}^F, \widetilde{T}_{1, \sigma(1)}^F) 
=
w({\pmb s}_{1, \sigma(1)}^F, T_{1, \sigma(1)}^F) \frac{1}{(T^*)^{{\pmb s}^*}},
$$
$$w(\widetilde{\pmb s}_{\sigma^{-1}(1), 1}^F, \widetilde{T}^F_{\sigma^{-1}(1), 1})
=
w({\pmb s}_{\sigma^{-1}(1), 1}^F, T^F_{\sigma^{-1}(1), 1}) \frac{1}{(T_*)^{{\pmb s}_*}}
$$
and
$$w(\widetilde{\pmb s}_{k, \sigma(k)}^F, \widetilde{T}_{k, \sigma(k)}^F) 
=
w({\pmb s}_{k, \sigma(k)}^F, T_{k, \sigma(k)}^F) \; ({\mathrm{for}\;} k\not=1, \sigma(k)\not=1),
$$
we get the following 
\begin{align}
w(\widetilde{\pmb s}, \widetilde{T})=\prod_{k=1}^N w({\pmb s}_{k, \sigma(k)}^F, T_k(\sigma)) \frac{1}{(T^*)^{{\pmb s}^*}T_*^{{\pmb s}_*}}
=\prod_{k=1}^N w(\widetilde{\pmb s}_{k, \sigma(k)}^F, \widetilde{T}_k(\sigma)) .
\end{align}
\end{proof}

Similarly with the discussion in the previous subsection,
for any given 
$\widetilde{T}_k(\sigma)\in {\rm SSYT}(\widetilde{p_k | q_{\sigma(k)}})$ ($1\leq k\leq N$),
we may construct a $\sigma$-tableau $\widetilde{T}=(T_* | T | T^*)$ by combining those $\widetilde{T}_k(\sigma)$ :
$$\widetilde{{\mathbb X}}:=\{(\sigma, \widetilde{T}) | \sigma\in {\frak S}_N, \widetilde{T} : \sigma \mbox{-tableau of shape} \; \widetilde{\lambda}\},
$$
where we call $\widetilde{T}$ a $\sigma$-tableau if it satisfies the conditions (I), (II), (III) in the previous subsection for $T$ .
From \eqref{firststep} and Lemma \ref{sigmatab}, we find

\begin{align*}
\det(\widetilde{\zeta}_{i,j}({\pmb s}_{i,j}^F))_{1\leq i,j\leq N}
=\sum_{\sigma\in {\frak S}_d}{\rm sgn}(\sigma) \sum_{\widetilde{T} : \sigma{\rm -tableau}}w({\pmb s}, \widetilde{T})
=\sum_{(\sigma, \widetilde{T})\in \widetilde{{\mathbb X}}}{\rm sgn}(\sigma) w({\pmb s}, \widetilde{T}).
\end{align*}
We denote the right-hand side of the above by $W_{\widetilde{\mathbb X}}$.

Let $\widetilde{T}\in {\rm SSYT}(\widetilde{\lambda})$.   In the same discussion as in the previous subsection, $\widetilde{T}$ is obviously an $\mathbbm{1}$-tableau,
while $\widetilde{T}$ is never a $\sigma$-tableau for any $\sigma\neq \mathbbm{1}$.

We devide
$\widetilde{\mathbb{X}}=\widetilde{{\mathbb X}}^{\dagger}\cup \widetilde{{\mathbb X}}^{\dagger\dagger}$, 
where $\widetilde{{\mathbb X}}^{\dagger}:=\{(\mathbbm{1}, \widetilde{T}) | \widetilde{T}\in {\rm SSYT}(\widetilde{{\lambda}})\}$
and ${\widetilde{\mathbb X}}^{\dagger\dagger}={\widetilde{\mathbb X}}\setminus\widetilde{{\mathbb X}}^{\dagger}$.
Since
\begin{align*}
W_{\widetilde{{\mathbb X}}^{\dagger}}:&=
\displaystyle{\sum_{(\mathbbm{1}, \widetilde{T})\in {\widetilde{\mathbb X}}^{\dagger}}
{\rm sgn}(\mathbbm{1})w({\widetilde{\pmb s}}, \widetilde{T})}
=\zeta_{(\widetilde{{\alpha | \beta}})},
\end{align*}
to prove \eqref{non-winged}
it is sufficient to show that $W_{\widetilde{\mathbb X}}=W_{{\widetilde{\mathbb X}}^{\dagger}}$. Hence,
it is enough to show that 
\begin{align}\label{daggerdagger}
W_{{\widetilde{\mathbb X}}^{\dagger\dagger}}:=
\sum_{(\sigma, \widetilde{T})\in {\widetilde{\mathbb X}}^{\dagger\dagger}}{\rm sgn}(\sigma)w({\widetilde{\pmb s}}, \widetilde{T})=0.
\end{align}
Indeed, the following lemma holds:
\begin{lem}
For any $(\sigma,\widetilde{T})\in {\widetilde{\mathbb X}}^{\dagger\dagger}$, 
there exists another element $(\sigma',\widetilde{T}')\in {\widetilde{\mathbb X}}^{\dagger\dagger}$ such that
\begin{align}
{\rm sgn}(\sigma)w(\widetilde{{\pmb s}},\widetilde{T})+{\rm sgn}(\sigma')w({\widetilde{\pmb s}},\widetilde{T}')=0.
\end{align}
\end{lem}

The proof goes through in the same way as in Lemma \ref{cancellation}, so we do not write down the details.

From the above lemma, we obtain that $W_{\widetilde{{\mathbb X}}^{\dagger\dagger}}=0$.
Therefore we now complete the proof of Theorem \ref{Giambelli1}.


\bigskip
\noindent
\textsc{Kohji Matsumoto}\\
Graduate School of Mathematics, \\
Nagoya University, \\
Furo-cho, Chikusa-ku, Nagoya, 464-8602, Japan \\
 \texttt{kohjimat@math.nagoya-u.ac.jp}\\
and\\
Center for General Education, \\
Aichi Institute of Technology, \\
1247 Yachigusa, Yakusa-cho, Toyota, 470-0392, Japan \\

\medskip

\noindent
\textsc{Maki Nakasuji}\\
Department of Information and Communication Science, Faculty of Science, \\
 Sophia University, \\
 7-1 Kioi-cho, Chiyoda-ku, Tokyo, 102-8554, Japan \\
 \texttt{nakasuji@sophia.ac.jp}\\
and\\
Mathematical Institute, \\
Tohoku University, \\
6-3 Aoba, Aramaki, Aoba-ku, Sendai, 980-8578, Japan \\

\end{document}